\documentclass[american]{article}
\usepackage[T1]{fontenc}
\usepackage[latin9]{inputenc}
\usepackage{color}
\usepackage{babel}
\usepackage{textcomp}
\usepackage{amsmath}
\usepackage{amsthm}
\usepackage{amssymb}
\usepackage[unicode=true,
 bookmarks=true,bookmarksnumbered=false,bookmarksopen=false,
 breaklinks=true,pdfborder={0 0 1},backref=false,colorlinks=false]
 {hyperref}
\hypersetup{pdftitle={A White Noise Approach to Stochastic Currents of Brownian Motion},
 pdfauthor={JosÃÂ© LuÃÂ­s da Silva and Herry P.  Suryawan},
 pdfsubject={Stochastic Currents of Brownian Motion},
 pdfkeywords={Brownian motion, White Noise Analysis,  Stochastic integrals},
 pdfborderstyle=}

\makeatletter
\theoremstyle{plain}
\newtheorem{thm}{\protect\theoremname}[section]
\theoremstyle{definition}
\newtheorem{defn}[thm]{\protect\definitionname}
\theoremstyle{definition}
\newtheorem{example}[thm]{\protect\examplename}
\theoremstyle{plain}
\newtheorem{cor}[thm]{\protect\corollaryname}
\theoremstyle{remark}
\newtheorem{rem}[thm]{\protect\remarkname}

\usepackage{babel}
\usepackage{xcolor}
\usepackage{pdfcolmk}\usepackage{bm}

\usepackage{pdfsync}\usepackage{bbm}
\newcommand{\zx}{\color{red}}
\newcommand{\xz}{\color{black}}

\numberwithin{equation}{section}

\providecommand{\corollaryname}{Corollary}
\providecommand{\definitionname}{Definition}
\providecommand{\examplename}{Example}
\providecommand{\remarkname}{Remark}
\providecommand{\theoremname}{Theorem}

\makeatother

\providecommand{\corollaryname}{Corollary}
\providecommand{\definitionname}{Definition}
\providecommand{\examplename}{Example}
\providecommand{\remarkname}{Remark}
\providecommand{\theoremname}{Theorem}

\begin{document}
\title{A White Noise Approach to Stochastic Currents of Brownian Motion}
\author{\textbf{Martin Grothaus}\\
 Department of Mathematics, University of Kaiserslautern,\\
 67653 Kaiserslautern, Germany\\
 Email: grothaus@mathematik.uni-kl.de \and \textbf{Herry Pribawanto
Suryawan}\\
 Department of Mathematics, Sanata Dharma University\\
 55281 Yogyakarta, Indonesia\\
 Email: herrypribs@usd.ac.id \and \textbf{Jos{\'e} Lu{\'\i}s da
Silva}\\
 CIMA, University of Madeira, Campus da Penteada,\\
 9020-105 Funchal, Portugal\\
 Email: joses@staff.uma.pt}
\date{}
\maketitle
 
\begin{abstract}
In this paper we study stochastic currents of Brownian motion $\xi(x)$,
$x\in\mathbb{R}^{d}$, by using white noise analysis. For $x\in\mathbb{R}^{d}\backslash\{0\}$
 and for $x=0\in\mathbb{R}$ we prove that the stochastic
current $\xi(x)$ is a Hida distribution. Moreover for $x=0\in\mathbb{R}^{d}$
 with $d>1$ we show that the stochastic current is not a Hida distribution.
\\
 \\
 \textbf{Keywords}: Stochastic currents, extended Skorokhod integral,
white noise analysis. 
\end{abstract}

\section{Introduction}

The concept of current is fundamental in geometric measure theory.
The simplest version of current is given by the functional 
\[
\varphi\mapsto\int_{0}^{T}(\varphi(\gamma(t)),\gamma'(t))_{\mathbb{R}^{d}}\,\mathrm{d}t,\quad0<T<\infty,
\]
in a space of vector fields $\varphi:\mathbb{R}^{d}\longrightarrow\mathbb{R}^{d}$
and $\gamma$ is a rectifiable curve in $\mathbb{R}^{d}$. Informally,
this functional may be represented via its integral kernel 
\[
\zeta(x)=\int_{0}^{T}\delta(x-\gamma(t))\gamma'(t)\,\mathrm{d}t,
\]
where $\delta$ is the Dirac delta distribution on $\mathbb{R}^{d}$.
The interested reader may find comprehensive account on the subject
in the books \cite{Federer1996,Morgan2016b}.

The stochastic analog of the current $\zeta(x)$ rises if we replace
the deterministic curve $\gamma(t)$, $t\in[0,T]$, by the trajectory
of a stochastic process $X(t)$, $t\in[0,T]$, in $\mathbb{R}^{d}$.
In this way, we obtain the following functional 
\begin{equation}
\xi(x):=\int_{0}^{T}\delta(x-X(t))\,\mathrm{d}X(t).\label{eq:stochastic-integral}
\end{equation}
The stochastic integral \eqref{eq:stochastic-integral} has to be
properly defined. Now we consider a $d$-dimensional Brownian motion
$B(t)$, $t\in[0,T]$, and the main object of our study is 
\begin{equation}
\xi(x)=\int_{0}^{T}\delta(x-B(t))\,\mathrm{d}B(t).\label{eq:Wick-type-Sintegral}
\end{equation}
In this work the stochastic integral \eqref{eq:Wick-type-Sintegral}
is interpreted as an extension of the Skorokhod integral developed
in \cite{HKPS93}. It coincides with the extension given by the adjoint
of the Malliavin gradient. There have been some other approaches to
study stochastic current, such as Malliavin calculus and stochastic
integrals via regularization, see \cite{Flandoli2005,Flandoli2009,Flandoli2010,Guo2014},
among others.

An initial study of the stochastic current (\ref{eq:Wick-type-Sintegral})
using white noise theory was done in \cite{Guo2013}. The authors
showed that $\xi(x)$ in (\ref{eq:Wick-type-Sintegral}) is well defined
as a Hida distribution for all $x\in\mathbb{R}^{d}$ and all dimensions
$d\in\mathbb{N}$. However the proof of Theorem 3.3 in \cite{Guo2013}
is not carefully written which lead the authors to an inaccurate conclusion.
In fact, for $x=0\in\mathbb{R}^{d}$, $d>1$, we show that $\xi(0)$
is not a Hida distribution. This is confirmed by first  orders  of the
chaos  expansion we obtained. Moreover, we got the impression that
the authors were not checking integrability of the integrand in \eqref{eq:stochastic-integral}.
Hence, they  cannot  apply Corollary \ref{cor:hida_integral} below.
We in turn could check the assumptions of Corollary \ref{cor:hida_integral},
below, for all  nonzero $x\in \mathbb{R}^d$, $d\in \mathbb{N}$, and for $x=0\in \mathbb{R}$. The aim of this paper is to fill this gap
and obtain kernels of first  orders  of the chaos expansion of $\xi(x)$.

The organization of the paper is as follows. Section \ref{sec:WNA}
provides some background of white noise analysis. In Section \ref{sec:Stochastic-Current}
we prove the main results of this paper on the existence of the Brownian
currents.

\section{Gaussian White Noise Analysis}

\label{sec:WNA}In this section we summarize pertinent results from
white noise analysis used throughout this work, and refer to \cite{HKPS93,KLPSW96,Kuo96}
and references therein for a detailed presentation.

\subsection{White Noise Space}

We start with the Gel'fand triple 
\[
S_{d}\subset L_{d}^{2}\subset S'_{d},
\]
where $S_{d}:=S(\mathbb{R},\mathbb{R}^{d})$, $d\in\mathbb{N}$, is
the space of vector valued Schwartz test functions, $S'_{d}$  is its
topological dual and the central Hilbert space $L_{d}^{2}:=L^{2}(\mathbb{R},\mathbb{R}^{d})$
of square integrable vector valued measurable functions. For any $f\in L_{d}^{2}$
given by $f=(f_{1},\ldots,f_{d})$ its norm is 
\[
|f|^{2}=\sum_{i=1}^{d}\int_{\mathbb{R}}|f_{i}(x)|^{2}\,\mathrm{d}x.
\]

Let $\mathcal{B}$ be the $\sigma$-algebra of cylinder sets on $S'_{d}$.
Since $S_{d}$ equipped with its standard topology is a nuclear space,
by Minlos' theorem there is a unique probability measure $\mu_{d}$
on $(S'_{d},\mathcal{B})$ with the characteristic function given
by 
\[
C(\varphi):=\int_{S'_{d}}e^{i\langle w,\varphi\rangle} \, \mathrm{d} \mu_{d}(w)=\exp\left(-\frac{1}{2}|\varphi|^{2}\right),\quad\varphi\in S_{d}.
\]
Hence, we have constructed the white noise probability space $(S'_{d},\mathcal{B},\mu_{d})$.
In the complex Hilbert space $L^{2}(\mu_{d}):=L^{2}(S'_{d},\mathcal{B},\mu_{d};\mathbb{C})$
a $d$-dimensional Brownian motion is given by 
\[
B(t,w)=(\langle w_{1},\eta_{t}\rangle,\ldots,\langle w_{d},\eta_{t}\rangle),\;w=(w_{1},\ldots,w_{d})\in S'_{d},\quad\eta_{t}:=1\!\!1_{[0,t)},\;t\ge0.
\]
In other words, $(B(t))_{t\ge0}$ consists of $d$ independent copies
of $1$-dimensional Brownian motions. For all $F\in L^{2}(\mu_{d})$
one has the Wiener-It{ô}-Segal chaos decomposition 
\[
F(w)=\sum_{n=0}^{\infty}\langle:w^{\otimes n}:,F_{n}\rangle,\quad F_{n}\in(L_{d,\mathbb{C}}^{2})^{\hat{\otimes}n},
\]
where $:w^{\otimes n}:\in(S'_{d,\mathbb{C}})^{\hat{\otimes}n}$ denotes
the $n$-th order Wick power of $w\in S'_{d,\mathbb{C}}$ and $\langle\cdot,\cdot\rangle$ denotes 
the dual pairing on $(S'_{d,\mathbb{C}})^{\otimes n}\times(S{}_{d,\mathbb{C}})^{\otimes n}$
which is a bilinear extension of \zx $(\cdot,\bar{\cdot})$, where $(\cdot,\cdot)$ \xz is the inner product on $(L_{d,\mathbb{C}}^{2})^{\otimes n}$
in the sense of a Gel'fand  triple. Here $V_{\mathbb{C}}$ denotes the
complexification of the real vector space $V$ and $\hat{\otimes}n$
denotes the $n$-th power symmetric tensor product. Note that $\langle:\cdot^{\otimes n}:,\cdot\rangle$,
$n\in\mathbb{N}_{0}$, in the second variable extends to $(L_{d,\mathbb{C}}^{2})^{\hat{\otimes}n}$
in the sense of an $L^{2}(\mu_{d})$ limit.

\subsection{Hida Distributions and Characterization}

By the standard construction with the Hilbert space $L^{2}(\mu_{d})$
as central space, we obtain the Gel'fand triple of Hida test functions
and Hida distributions. 
\[
(S_{d})\subset L^{2}(\mu_{d})\subset(S_{d})'.
\]
We denote the dual pairing between elements of $(S_{d})'$ and $(S_{d})$
by $\langle\!\langle\cdot,\cdot\rangle\!\rangle$. For $F\in L^{2}(\mu_{d})$
and $\varphi\in(S_{d})$, with kernel functions $F_{n}$ and $\varphi_{n}$,
respectively, the dual pairing yields 
\[
\langle\!\langle F,\varphi\rangle\!\rangle=\sum_{n=0}^{\infty}n!\langle F_{n},\varphi_{n}\rangle.
\]
This relation extends the chaos expansion to $\Phi\in(S_{d})'$ with
distribution valued kernels $\Phi_{n}\in(S'_{d,\mathbb{C}})^{\hat{\otimes}n}$
such that 
\[
\langle\!\langle\Phi,\varphi\rangle\!\rangle=\sum_{n=0}^{\infty}n!\langle\Phi_{n},\varphi_{n}\rangle,
\]
for every generalized test function $\varphi\in(S_{d})$ with kernels
$\varphi_{n}\in(S_{d,\mathbb{C}})^{\hat{\otimes}n}$, $n\in\mathbb{N}_{0}$.

Instead of repeating the detailed construction of these spaces we
present a characterization in terms of the $S$-transform. 
\begin{defn}
\label{def:S-transform}Let $\varphi\in S_{d}$ be given. We define
the Wick exponential by 
\[
e_{\mu_{d}}(\cdot,\varphi):=\frac{e^{\langle\cdot,\varphi\rangle}}{\mathbb{E}(e^{\langle\cdot,\varphi\rangle})}=C(\varphi)e^{\langle\cdot,\varphi\rangle}=\sum_{n=0}^{\infty}\frac{1}{n!}\langle:.^{\otimes n}:,\varphi^{\otimes n}\rangle\in(S_{d})
\]
and the $S$-transform of $\Phi\in(S_{d})'$ by 
\[
S\Phi(\varphi):=\langle\!\langle\Phi,e_{\mu_{d}}(\cdot,\varphi)\rangle\!\rangle.
\]
\end{defn}

\begin{example}
\label{exa:S-transform-WN}For $d\in\mathbb{N}$ the $S$-transform
of $d$-dimensional white noise $(W(t))_{t\ge0}$ is given by $SW(t)(\varphi)=\varphi(t)$,
for all $\varphi\in S_{d}$, $t\ge0$, see \cite{HKPS93}. Here $(W(t))_{t\ge0}$
is the derivative of $(B(t))_{t\ge0}$ as a Hida space valued process.
That is, each of its components takes values in $(S_{d})'$. 
\end{example}

\begin{defn}[$U$-functional]
\label{def:U-functional}A function $F:S_{d}\longrightarrow\mathbb{C}$
is called a $U$-functional if: 
\begin{enumerate}
\item For every $\varphi_{1},\varphi_{2}\in S_{d}$ the mapping $\mathbb{R}\ni\lambda\mapsto F(\lambda\varphi_{1}+\varphi_{2})\in\mathbb{C}$
has an entire extension to $z\in\mathbb{C}$. 
\item There are constants $ 0<  C_{1},C_{2}<\infty$ such that 
\[
\left|F(z\varphi)\right|\leq C_{1}\exp\big(C_{2}|z|^{2}\|\varphi\|^{2}\big),\quad\forall z\in\mathbb{C},\varphi\in S_{d}
\]
for some continuous norm $\|\cdot\|$ on $S_{d}$. 
\end{enumerate}
\end{defn}

We are now ready to state the aforementioned characterization result. 
\begin{thm}[cf.~\cite{KLPSW96}, \cite{PS91}]
\label{thm:charact-theorem} The $S$-transform defines a bijection
between the space $(S_{d})'$ and the space of $U$-functionals. In
other words, $\Phi\in(S_{d})'$ if and only if $S\Phi:S_{d}\to\mathbb{{C}}$
is a $U$-functional. 
\end{thm}

Based on Theorem \ref{thm:charact-theorem} a deeper analysis of the
space $(S_{d})'$ can be developed. The following corollary concerns
the Bochner integration of functions with values in $(S_{d})'$ (for
more details and proofs see e.g.~\cite{HKPS93}, \cite{KLPSW96},
\cite{PS91} for the case $d=1$). 
\begin{cor}\label{cor:hida_integral}Let $(\Omega,\mathcal{F},m)$ be a measure
space and $\lambda\mapsto\Phi_{\lambda}$ be a mapping from $\Omega$
to $(S_{d})'$. We assume that the $S$-transform of $\Phi_{\lambda}$
fulfills the following two properties: 
\begin{enumerate}
\item The mapping $\lambda\mapsto S\Phi_{\lambda}(\varphi)$ is measurable
for every $\varphi\in S_{d}$. 
\item The $U$-functional $S\Phi_{\lambda}$ satisfies 
\[
|S\Phi_{\lambda}(z\varphi)|\leq C_{1}(\lambda)\exp\left(C_{2}(\lambda)|z|^{2}\|\varphi\Vert^{2}\right),\quad z\in\mathbb{C},\varphi\in S_{d},
\]
for some continuous norm $\Vert\cdot\Vert$ on $S_{d}$ and for some
$C_{1}\in L^{1}(\Omega,m)$, $C_{2}\in L^{\infty}(\Omega,m)$. 
\end{enumerate}
Then 
\[
\int_{\Omega}\Phi_{\lambda}\,\mathrm{d}m(\lambda)\in(S_{d})'
\]
and 
\[
S\left(\int_{\Omega}\Phi_{\lambda}\,\mathrm{d}m(\lambda)\right)(\varphi)=\int_{\Omega}S\Phi_{\lambda}(\varphi)\,\mathrm{d}m(\lambda),\quad\varphi\in S_{d}.
\]
Moreover, the integral exists as a Bochner integral in some Hilbert
subspace of $(S_{d})'.$ 
\end{cor}

\begin{example}[Donsker's delta function]
\label{exa:S-transform-delta}As a classical example of a Hida distribution
we have the Donsker delta function. More precisely, the following
Bochner integral is a well defined element in $(S_{d})'$ 
\[
\delta(x-B(t))=\frac{1}{(2\pi)^{d}}\int_{\mathbb{R}^{d}}e^{i(\lambda,x-B(t))_{\mathbb{R}^{d}}}\,\mathrm{d}\lambda,\quad x\in\mathbb{R}^{d}.
\]
The $S$-transform of $\delta(x-B(t))$ for any $z\in\mathbb{C}$
and $\varphi\in S_{d}$ is given by 
\begin{equation}
S\delta(x-B(t))(z\varphi)=\frac{1}{(2\pi t)^{d/2}}\exp\left(-\frac{1}{2t}\sum_{j=1}^{d}(x_{j}-\langle z\varphi_{j},\eta_{t}\rangle)^{2}\right).\label{eq:S-Donsker}
\end{equation}
\end{example}

It is well known that the Wick product is a well defined operation
in Gaussian analysis, see for example \cite{KLS96}, \cite{HOUZ09}
and \cite{KSWY95}. 
\begin{defn}
\label{def:Wick_product}For any $\Phi,\Psi\in(S_{d})'$ the Wick
product $\Phi\Diamond\Psi$ is defined by 
\begin{equation}
S(\Phi\Diamond\Psi)=S\Phi\cdot S\Psi.\label{eq:Wick_product}
\end{equation}
Since the space of $U$-functional s  is an algebra, by Theorem\ \ref{thm:charact-theorem}
there exists an element $\Phi\Diamond\Psi\in(S_{d})'$ such that (\ref{eq:Wick_product})
holds. 
\end{defn}

\section{Stochastic Currents of Brownian Motion}

\label{sec:Stochastic-Current}In this section we investigate in the
framework of white noise analysis the following functional 
\begin{equation}
\varphi\mapsto\int_{0}^{T}(\varphi(B(t)),\mathrm{d}B(t))_{\mathbb{R}^{d}},\label{eq:current}
\end{equation}
on a given space of vector fields $\varphi:\mathbb{R}^{d}\longrightarrow\mathbb{R}^{d}$.
The functional (\ref{eq:current}) can be represented via its integral
kernel 
\[
\xi(x):=\int_{0}^{T}\delta(x-B(t))\,\mathrm{d}B(t),\quad x\in\mathbb{R}^{d}.
\]

We interpret the stochastic integral as an extended Skorokhod integral
\begin{align*}
 & \int_{0}^{T}\delta(x-B(t))\,\mathrm{d}B(t)\\
 & :=\left(\int_{0}^{T}\delta(x-B(t))\Diamond W_{1}(t)\,\mathrm{d}t,\ldots,\int_{0}^{T}\delta(x-B(t))\Diamond W_{d}(t)\,\mathrm{d}t\right)\\
 & =:(\xi_{1}(x),\ldots,\xi_{d}(x)),
\end{align*}
where $W=(W_{1},\ldots,W_{d})$ is the white noise process as in Example
\ref{exa:S-transform-WN}. If the integrand is
a square integrable function then this stochastic integral coincides
with the Skorokhod integral. In this interpretation, we call $\xi(x)$
stochastic currents of Brownian motion.

Below we show that $\xi(x)$, $x\in\mathbb{R}^{d}\backslash\{0\}$
is a well defined functional in $(S_{d})'$. From now on, $C$ is
a real constant whose value is immaterial and may change from line
to line. 
\begin{thm}
\label{thm:current-Hdistr} For $x\in\mathbb{R}^{d}\backslash\{0\}$,
$0<T<\infty$, the Bochner integral 
\begin{equation}
\xi_{i}(x)=\int_{0}^{T}\delta(x-B(t))\Diamond W_{i}(t)\,\mathrm{d}t\label{eq:current-bochner}
\end{equation}
is a Hida distribution and its $S$-transform at $\varphi\in S_{d}$
is given by 
\begin{equation}
S\left(\int_{0}^{T}\delta(x-B(t))\Diamond W_{i}(t)\,\mathrm{d}t\right)(\varphi)=\frac{1}{(2\pi)^{^{d/2}}}\int_{0}^{T}\frac{1}{t^{d/2}}e^{-\frac{|x-\langle\eta_{t},\varphi\rangle|_{\mathbb{R}^{d}}^{2}}{2t}}\varphi_{i}(t)\,\mathrm{d}t.\label{eq:S-transf-current-bm-d}
\end{equation}
\end{thm}

\begin{proof}
First we compute the $S$-transform of the integrand 
\[
(0,T]\ni t\mapsto\Phi_{i}(t):=\delta(x-B(t))\Diamond W_{i}(t).
\]
Using Definition \ref{def:Wick_product}, Example \ref{exa:S-transform-delta},  and Example
\ref{exa:S-transform-WN} for any $\varphi\in S_{d}$ we have 
\begin{align*}
t\mapsto S\Phi_{i}(t)(\varphi) & =S\big(\delta(x-B(t))\big)(\varphi)SW_{i}(t)(\varphi)\\
 & =\frac{1}{(2\pi t)^{d/2}}\exp\left(-\frac{1}{2t}|x-\langle\eta_{t},\varphi\rangle|_{\mathbb{R}^{d}}^{2}\right)\varphi_{i}(t),
\end{align*}
which is Borel measurable on $(0,T]$. Furthermore, for any $z\in\mathbb{C}$,
$t\in(0,T]$ and all $\varphi\in S_{d}$ we obtain 
\begin{eqnarray*}
 &  & \big|S\Phi_{i}(t)(z\varphi)\big|\\
 & \leq & \big|\frac{1}{\left(2\pi t\right)^{d/2}}\exp\left(-\frac{1}{2t}\big|x-\langle\eta_{t},z\varphi\rangle\big|_{\mathbb{R}^{d}}^{2}\right)z\varphi_{i}(t)\big|\\
 & \leq & \frac{1}{\left(2\pi t\right)^{d/2}}\exp\left(-\frac{1}{2t}|x|^{2}\right)\exp\left(\frac{1}{t}|x|\big|\langle\eta_{t},z\varphi\rangle\big|\right)\exp\left(\frac{1}{2t}|z|^{2}\big|\langle\eta_{t},\varphi\rangle\big|^{2}\right)\big|z\varphi_{i}(t)\big|\\
 & \leq & \frac{1}{\left(2\pi t\right)^{d/2}}\exp\left(-\frac{1}{2t}|x|^{2}\right)\exp\left(|x||z|\big|\varphi\big|_{\infty}\right)\exp\left(\frac{1}{2}|z|^{2}\big|\varphi\big|^{2}\right)\exp\left(|z|\big|\varphi\big|_{\infty}\right)\\
 & \leq & \frac{C}{\left(2\pi t\right)^{d/2}}\exp\left(-\frac{1}{2t}|x|^{2}\right)\exp\left(\frac{1}{2}|x|^{2}\right)\exp\left(\frac{1}{2}|z|^{2}\big|\varphi\big|^{2}\right)\exp\left(\frac{1}{2}|z|^{2}\big|\varphi\big|_{\infty}^{2}\right)\\
 & \leq & \frac{C}{\left(2\pi t\right)^{d/2}}\exp\left(-\frac{1}{2t}|x|^{2}\right)\exp\left(\frac{1}{2}|x|^{2}\right)\exp\left(\frac{1}{2}|z|^{2}\|\varphi\|^{2}\right),
\end{eqnarray*}
where $\|\cdot\|$ is the continuous norm on $S_{d}$ defined by $\|\varphi\|:=\sqrt{\big|\varphi\big|^{2}+\big|\varphi\big|_{\infty}^{2}}$.
The first factor $\frac{C}{\left(2\pi t\right)^{d/2}}\exp\left(-\frac{1}{2t}|x|^{2}\right)$
is integrable with respect to the Lebesgue measure $\mathrm{d}t$
on $[0,T]$. To be more precise, using the formula 
\[
\int_{u}^{\infty}y^{\nu-1}e^{-\mu y} \, \mathrm{d}  y=\mu^{-\nu}\Gamma\left(\nu,\mu u\right),\quad u>0,\mathrm{Re}(\mu)>0,
\]
where $\Gamma\left(\cdot,\cdot\right)$ is the complementary incomplete
gamma function, one can show that 
\[
\int_{0}^{T}\frac{1}{t^{d/2}}\exp\left(-\frac{1}{2t}|x|^{2}\right)\mathrm{d}t=2^{d/2-1}|x|^{2-d}\Gamma\left(\frac{d}{2}-1,\frac{|x|^{2}}{2T}\right).
\]
As the second factor $\exp\left(\frac{1}{2}|x|^{2}\right)\exp\left(\frac{1}{2}|z|^{2}\|\varphi\|^{2}\right)$
 is independent of $t\in(0,T]$, the result now follows from Corollary
\ref{cor:hida_integral}. 
\end{proof}

\begin{cor}
For $x=0$ and $d=1$ the stochastic current $\xi(0)$ is a Hida distribution, that is, the Bochner
integral 
\[
\xi(0)=\int_{0}^{T}\delta(B(t))\Diamond W(t)\,\mathrm{d}t
\]
is a Hida distribution. Moreover its $S$-transform at $\varphi\in S_{1}$
is given by 
\[
S\left(\int_{0}^{T}\delta(B(t))\Diamond W(t)\,\mathrm{d}t\right)(\varphi)=\frac{1}{\sqrt{2\pi}}\int_{0}^{T}\frac{1}{\sqrt{t}}e^{-\frac{\langle\eta_{t},\varphi\rangle^{2}}{2t}}\varphi(t)\,\mathrm{d}t.
\]
\end{cor}

\begin{proof}
By adapting the proof of Theorem \ref{thm:current-Hdistr} we obtain
for any $z\in\mathbb{C}$, $t\in(0,T]$ and all $\varphi\in S_{1}$
\[
|S\Phi(t)(z\varphi)|\le\frac{C}{\left(2\pi t\right)^{1/2}}\exp\left(\frac{1}{2}|z|^{2}\|\varphi\|^{2}\right).
\]
Since the function $(0,T]\ni t\mapsto t^{-1/2}$ is integrable with
respect to the Lebesgue measure, Corollary \ref{cor:hida_integral}
implies the statement of the corollary.
\end{proof}
\begin{rem}
We would like to comment on the chaos expansion of the stochastic
current of Brownian motion. To this end we identify the space $L_{d}^{2}$
with the Hilbert space $L^{2}(m):=L^{2}(E,\mathcal{B},m)$, where
$E:=\mathbb{R}\times\{1,\ldots,d\}$, $\mathcal{B}$ is the product
$\sigma$-algebra on $E$ of the Borel $\sigma$-algebra on $\mathbb{R}$
and the power set of $\{1,\ldots,d\}$ and $m=\mathrm{d}x\otimes\Sigma$
is the product measure of the Lebesgue measure on $\mathbb{R}$ and
the counting measure on $\{1,\ldots,d\}$. That is, for all $f,g\in L^{2}(m)$
we have 
\[
(f,g)_{L^{2}(m)}=\int_{E}f(x,i)g(x,i)\,\mathrm{d}m(x,i)=\sum_{i=1}^{d}\int_{\mathbb{R}}f(x,i)g(x,i)\,\mathrm{d}x.
\]
The $n$-th order chaos of a Hida distribution can be computed by
the $n$-th order derivative of its $S$-transform at the origin.
More precisely, for $\Psi\in(S_{d})'$ and $\varphi\in(S_{d})$ consider
the function 
\[
\mathbb{R}\ni s\mapsto U(s):=S\Psi(s\varphi)\in\mathbb{C}.
\]
Then the $n$-th order chaos $\Psi^{(n)}$ of $\Psi$ applied to $\varphi^{\otimes n}\in S_{d,\mathbb{C}}^{\hat{\otimes}n}$
is given by 
\[
\langle\Psi^{(n)},\varphi^{\otimes n}\rangle=\frac{1}{n!}\frac{\mathrm{d}^{n}}{\mathrm{d}s^{n}}U(s)\big|_{s=0},
\]
see \cite[Lemma~3.3.5]{O94}. In our situation we have 
\[
S\Phi_{i}(s\varphi)=\frac{1}{(2\pi t)^{d/2}}\exp\left(-\frac{1}{2t}|x-\langle\eta_{t},s\varphi\rangle|_{\mathbb{R}^{d}}^{2}\right)s\varphi_{i}(t),\quad i\in\{1,\ldots,d\}.
\]
Hence, for the stochastic currents of Brownian motion the first chaos
are given by 
\begin{align*}
\xi^{(0)}(x) & =(0,\ldots,0).\\
\xi_{i}^{(1)}(x) & =\left(\underbrace{0,\ldots,0}_{i-1},\frac{1}{(2\pi)^{d/2}}\int_{0}^{T}\frac{1}{t^{d/2}}e^{-\frac{|x|^{2}}{2t}}\delta_{t}\,\mathrm{d}t,0,\ldots,0\right).\\
\left(\xi_{i}^{(2)}(x)\right)_{j,k} & =-\frac{1}{4(2\pi)^{d/2}}\int_{0}^{T}\frac{1}{t^{d/2+1}}e^{-\frac{|x|^{2}}{2t}}\left(\mathrm{Id}_{ki}x_{j}\eta_{t}\otimes\delta_{t}+\mathrm{Id}_{ji}x_{k}\delta_{t}\otimes\eta_{t}\right)\mathrm{d}t,
\end{align*}
where $\mathrm{Id}$ denotes the identity matrix on $\mathbb{R}^{d}$
and $\delta_{t}$ denotes the Dirac distribution at $t>0$. Note that
for $x=0$ and $d>1$ the first chaos $\xi_{i}^{(1)}(0)$ is divergent,
hence in this case $\xi(0)$ cannot be a Hida distribution. In all
the other cases the integrals are well defined as Bochner integrals
in a suitable Hilbert subspace of ($S'_{d})^{\hat{\otimes}n}$, $n=0,1,2.$
This follows from the estimates for integrability we derived in the
proof of Theorem \ref{thm:current-Hdistr}. Indeed the estimates derived
to apply Corollary \ref{cor:hida_integral} imply that the integrands
$\Phi_{i}$, $1\le i\le d$, are Bochner integrable in some Hilbert
subspace $H_{-}$ of $(S'_{d})$ equipped with a norm $\|\cdot\|_{-}$,
see proof of \cite[Thm.~17]{KLPSW96}. More precisely,
there one shows that $\|\Phi_{i}\|_{-}$, $1\le i\le d$, is integrable.
That implies Bochner integrability of the kernels of $n$-th order
in the generalized chaos decomposition in a suitable Hilbert subspace
of $S'(\mathbb{R}^{n})$.
\end{rem}

\section{Conclusion and Outlook}

In this paper we give a mathematical rigorously meaning to the stochastic
current $\xi(x)$, $x\in\mathbb{R}^{d}\setminus\{0\}$  and $\xi(0)$,
$0\in\mathbb{R}$, of Brownian motion in the framework of white
noise analysis. On the other hand, for  $x=0\in\mathbb{R}^{d}$,
$d>1$,  we showed that $\xi$(0) is not a Hida distribution.
The first orders of the chaos expansion leave open whether the $\xi(x)$,
$x\in\mathbb{R}^{d}\backslash\{0\}$, are regular generalized functions
or even square integrable. That is, it is not obvious whether $\xi^{(n)}(x)\in(L_{d,\mathbb{C}}^{2})^{\hat{\otimes}n}$
or not for $x\in\mathbb{R}^{d}\backslash\{0\}$ and $n\in\mathbb{N}$.
In a future paper we plan to extend these results to a larger class
of stochastic processes, e.g., fractional Brownian motion and grey
Brownian motion.

\subsection*{Acknowledgments}

This work was partially supported by a grant from the Niels Hendrik
Abel Board and by the Center for Research in Mathematics and Applications
(CIMA) related with the Statistics, Stochastic Processes and Applications
(SSPA) group, through the grant UIDB/MAT/04674/2020 of FCT-Funda{\c c\~a}o
para a Ci{\^e}ncia e a Tecnologia, Portugal. We gratefully acknowledge
the financial support by the DFG through the project GR 1809/14-1.

\global\long\def\etalchar#1{$^{#1}$}%

\end{document}